\documentclass[letter,12pt]{article}

\usepackage[english]{babel}
\usepackage{amsthm}
\usepackage{amsfonts}
\usepackage{amsmath}
\usepackage{anysize}
\usepackage{bbm}
\usepackage{amssymb}
\numberwithin{equation}{section}

\usepackage{abstract}
\marginsize{2.6cm}{2.6cm}{1cm}{2cm}

\newtheorem*{theorem*}{Theorem}
\newtheorem{theorem}{Theorem}[section]

\newtheorem{lemma}[theorem]{Lemma}
\newtheorem{corollary}[theorem]{Corollary}
\newtheorem{remark}[theorem]{Remark}

\newtheorem{proposition}[theorem]{Proposition}
\newtheorem{conjecture}[theorem]{Conjecture}

\def\del{\partial}
\def\dbar{\bar\partial}
\def\ddbar{\del\dbar}

\newcommand{\Cc}{\mathcal{C}}
\def\del{\partial}

\DeclareMathOperator{\Ric}{Ric}

\def\Ent{{\rm Ent}}
\def\Ec{\mathcal{E}}

\newcommand{\Kenergy}{E}
\newcommand{\AM}{{\rm AM}}
\newcommand{\setdef}{\ \mid \ }
\newcommand{\Hnormalize}{\mathcal{H}_{0}}
\newcommand{\Homega}{\mathcal{H}_{\omega}}
\newcommand{\Tr}{{\rm Tr}}
\newcommand{\AubinI}{I}
\newcommand{\PSH}{{\rm PSH}}
\newcommand{\AUT}{{\rm Aut}}

\usepackage{hyperref}
\hypersetup{
    bookmarks=true,         
    unicode=false,          
    pdftoolbar=true,       
    pdfmenubar=true,       
    pdffitwindow=false,    
    pdfstartview={FitH},   
    pdftitle={Regularity of weak minimizers of the K energy},    
    pdfauthor={R.J. Berman, T. Darvas, C.H. Lu},     
    colorlinks=true,       
   linkcolor=black,          
    citecolor=black,        
    filecolor=black,      
    urlcolor=black}           

\title{Regularity of weak minimizers of the K-energy and applications to properness and K-stability}
\author{Robert J. Berman, Tam\'as Darvas, Chinh H. Lu}
\date{\vspace{-0.2in}}
\begin{document}
\maketitle
\begin{abstract}
Let $(X,\omega)$ be a compact K\"ahler manifold and $\mathcal H$ the space of K\"ahler metrics cohomologous to $\omega$. If a csck metric exists in $\mathcal H$, we show that all finite energy minimizers of the extended K-energy are smooth csck metrics, partially confirming a conjecture of Y.A. Rubinstein and the second author. As an immediate application, we obtain that existence of a csck metric in $\mathcal H$ implies J-properness of the K-energy, thus confirming one direction of a conjecture of Tian. Exploiting this properness result we prove that an ample line bundle $(X,L)$ admitting a csck metric in $c_1(L)$ is $K$-polystable. When the automorphism group is finite, the properness result, combined with a result of Boucksom-Hisamoto-Jonsson, also implies that  $(X,L)$ is uniformly K-stable. 
\end{abstract}

\section{Introduction and main results}
Let $(X,J,\omega)$ be a compact connected K\"ahler manifold. By 
\[
\Homega = \{v \in \Cc^\infty(X) \setdef  \omega_{v}:=\omega + i\ddbar v > 0  \}
\] 
we denote the space of K\"ahler potentials. By the $\ddbar$-lemma of Hodge theory, up to a constant, this space is in a one-to-one correspondence with $\mathcal H$, the space of K\"ahler metrics cohomologous to $\omega$. The problem of finding canonical metrics in $\mathcal H$ goes back to Calabi in 50's. In this work we will point necessary conditions under which $\mathcal H$ admits constant scalar curvature K\"ahler (csck) metrics, in terms of energy properness.

We now elaborate on the terminology necessary to state our main results. To have a one-to-one correspondence between potentials and metrics, we consider  the space
$$\Hnormalize := \Homega \cap \AM^{-1}(0),$$ 
and we always work on the level of potentials unless specified otherwise (for the definition of $\AM$ see \eqref{eq: def MA energy} below). The connected Lie group of holomorphic automorphisms   
$$G:=\textup{Aut}_0(X,J)$$ 
acts naturally on $\mathcal H$ via pullbacks, hence it also acts on $\Hnormalize$ (see \cite[Section 5.2]{DR17} for a precise description of this action on the level of potentials). 

Motivated by results and ideas in conformal geometry, in the 90's Tian introduced the notion of ``J-properness" on $\Homega$ \cite[Definition 5.1]{t1} in terms of Aubin's nonlinear
 energy functional $J_\omega$ and the Mabuchi K-energy 
 $\Kenergy$. This condition says that for any $u_j \in \Homega$ we have 
\begin{equation}\label{eq: Jpropeness_of_K}
J_\omega(u_j)  \to \infty \ \ \ \textup{ implies } \ \ \ \Kenergy(u_j) \to \infty.
\end{equation}
We refer to Section 2 for the precise definitions of $J_\omega$ and $\Kenergy$. 

Tian conjectured that existence of constant scalar curvature K\"ahler (csck) metrics in $\Homega$ should be equivalent to J-properness of the K-energy $\Kenergy$ \cite[Remark 5.2]{t1},\cite{tianbook} and this was proved for Fano manifolds with $G$ trivial \cite{t2,tz}.  In \cite[Theorem 1]{pssw} the ``strong form" of the J-properness condition \eqref{eq: Jpropeness_of_K} was obtained, confirming another conjecture of Tian from \cite{t2} (for Fano manifold with trivial $G$), saying that the K-energy grows at least linearly with respect to the J-functional. This stronger form has been later adopted in the literature, sometimes referred to as ``coercivity". 

When $G$ is non-trivial it was known that the conjecture cannot, in general, hold as stated above and numerous modifications were proposed by Tian (see \cite[Conjecture 7.12]{tianbook}, \cite{t4}). In \cite{DR17}, Y.A. Rubinstein and the second named author disproved one of these conjectures, proved the remaining ones for general Fano manifolds, and the following conjecture was stated for general K\"ahler manifolds:

\begin{conjecture}[Conjecture 2.8 in \cite{DR17}] \label{conj: Tianmodified} Suppose $(X,\omega)$ is a K\"ahler manifold. There exists a csck metric cohomologous to $\omega$ if and only if for some $C,D>0$ we have 
\[
\Kenergy(u) \geq C \inf_{g \in G}J_\omega(g.u) - D, \ \ u \in \Homega.
\]
\end{conjecture}

This ``modified properness conjecture'' thus reduces to Tian's original prediction in case $G$ is trivial and was originally stated for Fano manifolds by Tian himself \cite{t4}. It was proved in this context (of Fano manifolds) in \cite[Theorem 2.4]{DR17}, and this paper also linked the resolution of the general conjecture to a regularity question on weak minimizers of the K-energy that we elaborate now. 

We denote by $(\mathcal E^1,d_1)$ the metric completion of $\Homega$ with respect to the $L^1$-type Mabuchi path length metric $d_1$. We refer to Sections 2.1-2.2 for more precise details about this metric structure introduced in \cite{da2}.  The point of connection with the questions investigated here is the fact that  $d_1$ metric growth is comparable to $J_\omega$ \cite[Proposition 5.5]{DR17},  and we refer to \cite[Section 4, Section 5]{DR17} for a more detailed exposition on how the $d_1$-metric geometry relates to J-properness.  
Let us now state the regularity conjecture of \cite{DR17} (see \cite[Conjecture 2.9]{DR17}) and the theorem that connects it to Conjecture \ref{conj: Tianmodified} above:

\begin{conjecture}[Conjecture 2.9 in \cite{DR17}] \label{conj: DRconj} Suppose $(X,\omega)$ is a compact K\"ahler manifold. The minimizers of the extended K-energy $\Kenergy: \mathcal E^1 \to (-\infty,+\infty]$ are smooth csck metrics.
\end{conjecture}

\begin{theorem}[Theorem 2.10 in \cite{DR17}] \label{thm: DRthm} Conjecture \ref{conj: DRconj} implies Conjecture \ref{conj: Tianmodified}. 
\end{theorem}

Our first main result partially confirms Conjecture \ref{conj: DRconj} and also a less general conjecture of X.X. Chen \cite[Conjecture 6.3]{c5}:

\begin{theorem} \label{thm: reguarity theorem} Suppose $(X,\omega)$ is a csck manifold. If $v \in \mathcal E^1$ minimizes the extended K-energy $\Kenergy: \mathcal E^1 \to (-\infty,+\infty]$, then $v$ is a smooth csck potential. In particular there exists $g \in G$ such that $g^* \omega_v = \omega$. 
\end{theorem}

The last claim follows from the uniqueness result of \cite{bb}. Using this result and Theorem \ref{thm: DRthm} we immediately obtain one direction of Conjecture \ref{conj: Tianmodified}:

\begin{theorem} \label{thm: TianCor1} Suppose $(X,\omega)$ is a csck manifold. Then for some $C,D>0$ we have 
\begin{equation}\label{eq: K_is_J_Gproper}
\Kenergy(u) \geq C \inf_{g \in G}J_\omega(g.u) - D, \ \ u \in \Homega.
\end{equation}
\end{theorem}

The proof of Theorem \ref{thm: reguarity theorem} relies on the $L^1$-Mabuchi geometry of $\Homega$ introduced in \cite{da1,da2}, the finite energy pluripotential theory of \cite{bbegz,GZ07} and the convexity methods of \cite{DR17,bdl} and \cite{bb}. Realizing that the metric geometry of $\Homega$ and J-properness should be related seems to have first appeared in \cite[Conjecture 6.1]{c5}, but this work rather proposed the use of the $L^2$-Mabuchi metric on $\Homega$.

As a consequence of Theorem \ref{thm: TianCor1} and the techniques of \cite{t2,b} we obtain a result on K-polystability, originally proved by Mabuchi (\cite[Main Theorem]{ma2} see also \cite{ma1}), using a completely different argument. Slightly less  general, or different flavor results were obtained by Stoppa, Stoppa-Sz\'ekelyhidi, Sz\'ekelyhidi \cite[Theorem 1.2]{sto}, \cite[Theorem 1.4]{stosz}, \cite[Theorem A]{sz} and others. We recall the relevant terminology in the last section of the paper.

\begin{theorem} \label{thm: Kpolystable}Suppose $L \to X$ is a positive line bundle. If there exists a csck metric in the class  $c_1(L)$, then $(X,L)$ is K-polystable.
\end{theorem}

The idea of proving K-stability via properness goes back to Tian's seminal paper \cite{t2}.  The main point of our approach, involving geodesic rays, is to generalize the findings of \cite{b} from the Fano case.

In case the group $G$ is trivial,  the results in \cite{bbj,bhj,DeR} show that properness  implies uniform K-stability in the $L^1$-sense (for terminology, see \cite{bbj,bhj,DeR} and references therein). Thus, as a consequence of Theorem \ref{thm: TianCor1} we obtain the following:
 
\begin{corollary} Assume that $(X,L)$ is a positive line bundle and $G$ is trivial. If there exists a csck metric in $c_1(L)$, then $(X,L)$ is uniformly K-stable. 
\end{corollary}

\paragraph*{Further relations to previous results.} We end the introduction with a brief (but by no means complete) discussion about further relations to previous results. Much work has been done on Tian's properness conjectures in the case when the K\"ahler class is anti-canonical and we refer to \cite{DR17} for a detailed historical account. To our knowledge, in the case of csck metrics, excluding perhaps the particular case of toric K\"ahler manifolds, no partial results are known, even when $G$ is trivial. 

Conjecture \ref{conj: DRconj} is known to be true in case $(X,\omega)$ is Fano \cite{be2}. Proving the reverse direction of Conjecture \ref{conj: Tianmodified} via Theorem \ref{thm: DRthm} seems to require further progress on the theory of fourth order partial differential equations and seems to be out of reach for the moment. 

Using the ideas of \cite{bb,DR17}, it is likely that different versions of the above  properness theorem can be obtained assuming existence of extremal or soliton/edge type csck metrics, but also for different spaces of potentials, mimicking \cite[Theorem 2.1, Theorem 2.11, Theorem 2.12]{DR17}. We refer to \cite[Remark 4.14]{bdl} for a result on twisted csck metrics.

Our K-stability results fit into a circle of ideas surrounding the Yau--Tian--Donaldson conjecture on a polarized manifold $(X,L)$, saying that the first Chern class of $L$ contains a K\"ahler metric with constant scalar curvature if and only if $(X,L)$ is stable in an appropriate sense, inspired by Geometric Invariant Theory. In the formulation introduced by Donaldson \cite{do1} the stability in question was formulated as K-polystability, but in view of an example in \cite{acgt} there is widespread belief that the notion of K-(poly)stability has to be strengthened (unless $X$ is Fano and $L$ is the anti-canonical polarization). In case $G$ is trivial, uniform K-stability was introduced in the thesis of Sz\'ekelyhidi (see also \cite{sz,Der,bhj0}) to provide such a stronger notion. In light of the recent variational approach to the Yau--Tian--Donaldson conjecture introduced in \cite{bbj} it seems that one of the main analytic hurdles in proving that uniform K-stability conversely implies the existence of a constant scalar curvature metric is the general form of the regularity conjecture alluded to above. 
Finally, it seems likely that our proof of K-polystability can be extended to the transcendental setting considered very recently in \cite{SD,DeR} but we will not go further into this here. In the case when $G$ is trivial, properness does imply uniform K-stability also in the transcendtal setting, as shown in \cite{DeR}.

\paragraph{Acknowledgments.} The first named author was supported by the Swedish Research Council, the European Research Council and 
the Knut and Alice Wallenberg foundation.  The second named author has been partially supported by BSF grant 2012236 and NSF grant DMS-1610202. 

The first named author is grateful to S\'ebastien Boucksom and Mattias Jonsson for discussions and for sharing the preprint \cite{bhj}. The second named author would 
like to thank L. Lempert, Y. Rubinstein for enlightening discussions on the topic of the paper and for making suggestions on how to improve the paper. The second named author would also like to thank G. Tian for explaining him how the techniques of \cite{t2} can be used to give another proof of Theorem \ref{thm: Kpolystable}.
Lastly, we would like to thank the anonymous referee for suggesting numerous changes that greatly improved the presentation of the paper.

\section{Preliminaries} 
We recall several known results that are needed in the present paper.  We refer the reader to \cite{da1,da2},  \cite[Section 2]{bdl} and \cite[Section 5.1]{DR17} for more information. Below we will follow the notations of \cite{DR17}.   

Fix a compact connected K\"ahler manifold $(X,\omega)$  of dimension $n$. A function $u: X \rightarrow \mathbb{R}\cup \{-\infty\}$ is called quasi-plurisubharmonic if locally $u= \rho + \varphi$, where $\rho$ is smooth and $\varphi$ is a plurisubharmonic function. We say that $u$ is $\omega$-plurisubharmonic ($\omega$-psh for short) if it is quasi-plurisubharmonic and $\omega_u:=\omega+i\ddbar u \geq 0$ in the weak sense of currents on $X$. We let $\PSH(X,\omega)$ denote the space of all $\omega$-psh functions on $X$.
Clearly $\mathcal H_\omega \subset \textup{PSH}(X,\omega)$, and this latter space also hosts metric completions of $\mathcal H_\omega$. We recall this below, along with standard terminology and results from finite energy pluripotential theory, that will be essential for the rest of this paper.

\subsection{The finite energy space $\mathcal{E}^1$}
In our analysis below we will mainly work with singular potentials. 
For  bounded $\omega$-psh functions $u_1,...,u_n$ the mixed Monge-Amp\`ere measures $\omega_{u_1} \wedge... \wedge \omega_{u_n}$ were introduced by Bedford-Taylor \cite{BT76,BT82}, generalizing the usual wedge product of smooth forms. As observed in \cite[Secton 1.1]{GZ07}, for an $n$-tuple of (possibly unbounded) $\omega$-psh functions $u_1,...,u_n$   the sequence of measures 
\[
\mathbbm{1}_{\bigcap_{k=1}^n\{u_k>-j\}} \omega_{\max(u_1,-j)} \wedge \cdots \wedge \omega_{\max(u_n,-j)}
\]
is non-decreasing in $j$, and we set $\omega_{u_1} \wedge \cdots \wedge  \omega_{u_n}$ to be the ``strong limit'' of these measures (for more details we refer to \cite[Section 1.1]{GZ07}). When $u_1=\cdots =u_n=u$ we simply set $\omega_u^n:= \omega_u \wedge \cdots \wedge \omega_u$.  By construction, $\omega_u^n$ is a positive Borel measure on $X$ whose mass can take any value in $[0,V]$, where $V=\int_X \omega^n$. The class $\Ec$  consists of functions $u\in \PSH(X,\omega)$ such that $\omega_u^n$ has full mass, i.e. $\int_X \omega_u^n=V$.  The class $\Ec^1$ consists of functions $u\in \Ec$ such that $\int_X |u| \, \omega_u^n<+\infty$. 
We refer the readers to \cite{GZ07} for a detailed study of this finite energy class. 

\subsection{The geodesic metric space $(\Ec^1,d_1)$}
Given two K\"ahler potentials $u_0,u_1\in \Homega$ we define 
\[
d_1(u_0,u_1) := \inf\left \{\int_0^1 \int_X |\dot{u}_t| \omega_{u_t}^n dt\right \},
\]
where the infimum is taken over all smooth curves $u(t,x)=u_t(x) \in \Cc^{\infty}([0,1]\times X)$ such that $u_t\in \Homega$, for all $t\in[0,1]$. Here $\dot{u}_t$ is the $t$-derivative of $u$. As shown in \cite{da1,da2}  $d_1$ is a bona fide metric on $\mathcal H$. A curve $[0,1] \ni t \mapsto u_t \in \PSH(X,\omega) \cap L^{\infty}(X)$ is called a weak geodesic segment if the complexified curve 
\[
Y:= ([0,1]\times \mathbb{R}) \times X \ni (z,x) \mapsto U(z,x):= u({\rm Re}(z), x)
\]
satisfies  $\pi_2^* \omega + dd^c U \geq 0$ and 
\[
(\pi_2^* \omega +dd^c U)^{n+1} =0
\]
in the sense of measures in $Y$. By the main results of \cite{ch0}  if $u_0,u_1$ are in $\Homega$ then there exists a unique weak geodesic segment $t \to u_t$ connecting $u_0$ and $u_1$, with bounded Laplacian on the product $Y$. 

By approximation one can define the finite energy geodesics connecting any $u_0,u_1\in \Ec^1$. As shown in \cite{da1,da2}, the distance $d_1$ can be extended to $\Ec^1$ making $(\Ec^1,d_1)$ a complete geodesic metric space which is the completion of $(\Homega,d_1)$ \cite[Theorem 2]{da2}. For details on this we refer the interested reader to the original papers, as well as the comprehensive recent survey \cite[Chapter 3]{Da18}.

\subsection{The energy functionals}

In this section we recall well known facts from the literature about the canonical functionals of K\"ahler geometry, and their extensions to the finite energy space $\mathcal E^1$.

\subsubsection*{The Monge-Amp\`ere energy and its contracted version}
The Monge-Amp\`ere  energy (often referred to as Aubin-Mabuchi or Aubin-Yau energy) is defined as 
\begin{equation}\label{eq: def MA energy}
\AM(u)= \frac{1}{(n+1)V} \sum_{j=0}^{n} \int_X u\, \omega_u^{j} \wedge \omega^{n-j}, u\in \PSH(X,\omega)\cap L^{\infty}(X). 
\end{equation} 
Given a smooth positive closed $(1,1)$-form $\chi$ the $\chi$-contracted version of the Monge-Amp\`ere energy is defined as 
\begin{equation}
	\label{eq: def contracted MA energy}
	\AM_{\chi}(u) = \frac{1}{nV}\sum_{j=0}^{n-1}\int_X u\omega_u^{j} \wedge \omega^{n-1-j}\wedge  \chi, \ u \in \PSH(X,\omega) \cap L^{\infty}(X). 
\end{equation}

Formally, the first order variation of the two energy functionals $\AM$ and $\AM_{\chi}$ along a curve $t \to u_t$ is given by similar formulas: 
\begin{equation}
	\label{eq: derivative of AM}
	\frac{d}{dt}\AM(u_t)= \frac{1}{V} \int_X \dot u_t\, \omega_{u_t}^n\ \text{and}\ \frac{d}{dt}\AM_\chi(u_t)= \frac{1}{V} \int_X \dot u_t\, \chi \wedge \omega_{u_t}^{n-1}. 
\end{equation}
As a consequence,  $v \in \Hnormalize$ is a critical point of $\AM_\chi$ if  and only if 
\[
\Tr^{\omega_v} (\chi):= n \frac{\chi \wedge \omega_v^{n-1}}{\omega_v^n}
\]
is constant. Applying  \eqref{eq: derivative of AM} for $u_t:= tu+(1-t)v$ with $u,v \in \PSH(X,\omega)\cap L^{\infty}(X)$ an elementary integration recovers the well known formulas 
\begin{flalign}
	\AM(u) -\AM(v)& = \frac{1}{(n+1)V}\sum_{k=0}^n   \int_X (u-v) \omega_u^k\wedge \omega_v^{n-k} \label{eq: AM u - AM v} \\
	\AM_{\chi}(u)- \AM_{\chi}(v) & = \frac{1}{nV}\sum_{k=1}^n   \int_X (u-v) \omega_u^k\wedge \omega_v^{n-k-1}\wedge \chi. \label{eq: AM chi u -v}
\end{flalign}
An integration by parts then shows that, for $u,v\in \PSH(X,\omega)\cap L^{\infty}(X)$,
\begin{flalign}
\int_X (u-v) \omega_u^{k} \wedge \omega_v^{n-k} &  \leq \int_X (u-v) \omega_u^{\ell} \wedge \omega_v^{n-\ell}, 0 \leq \ell \leq k\leq n, \label{eq: integration by parts mixed MA} \\
	\frac{1}{V}\int_X (u-v) \omega_u^n &\leq  \AM(u) -\AM(v) \leq \frac{1}{V} \int_X (u-v) \omega_v^n \label{eq: inequality for AM} \\
	\frac{1}{V} \int_X(u-v) \omega_u^{n-1}\wedge \chi &\leq  \AM_{\chi}(u) -\AM_{\chi}(v) \leq \frac{1}{V}\int_X (u-v) \omega_v^{n-1}\wedge \chi. \label{eq: inequality for AM chi}
\end{flalign}
In particular $\AM$ and $\AM_{\chi}$ are non-decreasing and one can extend these functionals to the whole space $\PSH(X,\omega)$: 
\begin{flalign*}
\AM(u) &:= \inf \{\AM(v) \setdef u\leq v \in \PSH(X,\omega)\cap L^{\infty}(X)\}, \\
\AM_{\chi}(u)& := \inf \{\AM_{\chi}(v) \setdef u\leq v \in \PSH(X,\omega)\cap L^{\infty}(X)\}.
\end{flalign*}
Given $u\in \PSH(X,\omega)$, it was shown in \cite[Lemma 2.7 and Proposition 2.8]{bbgz} that $u\in \Ec^1$ if and only if $\AM(u)$  is finite, which  also implies that $\AM_{\chi}$ is finite.  It was shown in \cite[Lemma 4.15]{da2} and \cite[Section 4]{bdl} that the energy functionals $\AM,\AM_{\chi}$ can be extended to $(\Ec^1,d_1)$ as $d_1$-Lipschitz functionals. Using an approximation argument together with \cite[Lemma 5.2]{da2} one can argue that all the above identities and inequalities hold for $u,v\in \Ec^1$. For details see \cite[Section 4]{bdl}, as well as \cite{Da18}.

\subsubsection*{The Aubin $\AubinI$ functional}
We recall the definition of the $\AubinI$ functional introduced by Aubin \cite[Section III]{Au} (and extended to $\mathcal E^1$ in \cite[Section 1.4]{bbegz}):
\[
\AubinI(u_0,u_1)= \frac{1}{V}\int_X (u_0-u_1)(\omega^n_{u_1} - \omega^n_{u_0}), \ \ u_0,u_1 \in \mathcal E^1.
\]

From the definition it is clear that $\AubinI$ is symmetric and invariant under adding constants. By an integration by parts we see that $\AubinI$ is non-negative. On the other hand, \cite[Theorem 3]{da2} implies that $I$ is $d_1$-continuous in both components.

Also, \cite[Lemma 1.9]{bbegz} implies that $I$ is non-degenerate, i.e., $I(u_0,u_1)=0$ if and only if $u_0 = u_1 + c$ for some $c \in \Bbb R$.

\begin{lemma}
	\label{lem: compare I and AM}
	For every $u_0\in \Hnormalize$ and $u_1\in \Ec^1\cap \AM^{-1}(0)$ the following estimates hold:
\begin{equation}\label{eq: JIdoubleestimate}
\frac{1}{n(n+1)}\AubinI(u_0,u_1)\leq \AM_{\omega_{u_0}}(u_1)- \AM_{\omega_{u_0}}(u_0)\leq \AubinI(u_0,u_1).
\end{equation}
\end{lemma}
\begin{proof}
	Fix $u_0 \in \Hnormalize$ and $u_1\in \Ec^1\cap \AM^{-1}(0)$.  It follows from \eqref{eq: inequality for AM chi}  and \eqref{eq: integration by parts mixed MA} that 
\begin{flalign}
&\AM_{\omega_{u_0}}(u_1)- \AM_{\omega_{u_0}}(u_0) = \frac{1}{nV}\sum_{k=1}^n \int_X (u_1-u_0) \omega_{u_0}^k \wedge \omega_{u_1}^{n-k}, \label{eq: compare I and AM chi 1}  \\
 & \AM_{\omega_{u_0}}(u_1)- \AM_{\omega_{u_0}}(u_0) \leq \frac{1}{V}\int_X (u_1-u_0)\omega_{u_0}^n. 	\label{eq: compare I and AM chi 2} 
\end{flalign}
Since $\AM(u_0)=\AM(u_1)=0$ it follows from \eqref{eq: AM chi u -v} and \eqref{eq: inequality for AM} that 
\begin{flalign}
	&\int_X (u_1-u_0) \omega_{u_1}^n  \leq  0 \leq \int_X (u_1-u_0) \omega_{u_0}^n, \label{eq: compare I and AM chi 3}\\
	&\sum_{k=0}^n \int_X (u_1-u_0) \omega_{u_0}^k \wedge \omega_{u_1}^{n-k} =0. \label{eq: compare I and AM chi 4}
\end{flalign}
Thus the second inequality in \eqref{eq: JIdoubleestimate} follows from \eqref{eq: compare I and AM chi 2} and \eqref{eq: compare I and AM chi 3}.  From \eqref{eq: compare I and AM chi 4} and \eqref{eq: compare I and AM chi 1} it follows that
\begin{equation}\label{eq: compare I and AM chi 5}
	\AM_{\omega_{u_0}}(u_1) -\AM_{\omega_{u_0}}(u_0) = \frac{-1}{nV} \int_X (u_1-u_0) \omega_{u_1}^n.  
\end{equation}
From \eqref{eq: compare I and AM chi 4}, \eqref{eq: AM u - AM v} and \eqref{eq: integration by parts mixed MA} it follows that 
\begin{flalign}\label{eq: compare I and AM chi 6}
	n \int_X (u_1-u_0) \omega_{u_1}^n + \int_X (u_1-u_0) \omega_{u_0}^n \leq \AM(u_1)-\AM(u_0) =0. 
\end{flalign}
From \eqref{eq: compare I and AM chi 6} and the definition of the functional $\AubinI$ we see that 
\begin{flalign*}\label{eq: compare I and AM chi 7}
\AubinI(u_0,u_1)  & = \frac{1}{V}\left ( \int_X (u_1-u_0) \omega_{u_0} + n  \int_X (u_1-u_0) \omega_{u_1}\right) + \frac{-(n+1)}{V} \int_X (u_1-u_0)\omega_{u_1}^n\\  
& \leq 	\frac{-(n+1)}{V} \int_X (u_1-u_0)\omega_{u_1}^n.  
\end{flalign*}
This combined with \eqref{eq: compare I and AM chi 5} gives the first inequality in \eqref{eq: JIdoubleestimate}. 
\end{proof}

We also recall that, as shown in \cite[Theorem 1.8]{bbegz}, the $\AubinI$-functional satisfies a quasi-triangle inequality, i.e., there exists a constant $c_n>0$ depending only on $n$ such that 
\begin{equation}
	\label{eq: quasi triangle for Aubin I}
	c_n\AubinI(u_0,u_1) \leq  \AubinI(u_0,v) +  \AubinI(u_1,v) , \ u_0,u_1,v\in \Ec^1. 
\end{equation}

\subsubsection*{The $J$ functional}
The $J$ functional, introduced by Aubin \cite[Section III]{Au},  is defined as 
\[
J_{\omega}(u) :=  \frac{1}{V}\int_X u\, \omega^n - \textup{AM}(u), \ u \in \Ec^1,
\]
If $u\in \Ec^1\cap \AM^{-1}(0)$ then we have $J_{\omega}(u)= V^{-1} \int_X u\,  \omega^n$. It follows from \eqref{eq: inequality for AM} that $J_{\omega}(u)\geq 0$, for all $u\in \Ec^1$.  It is a classical fact that (see for example \cite[Inequality (2.7) page 193]{bbgz})  
\[
\frac{1}{n+1} \AubinI(u,0) \leq J_{\omega}(u) \leq \AubinI(u,0),  \ u \in \Ec^1. 
\]
In the literature this functional is mainly denoted by $J$ (see \cite{DR17}). We use the notation $J_{\omega}$ to distinguish this functional from the complex structure $J$ of $X$.

\subsubsection*{The (twisted) K-energy}
Fix a closed smooth positive $(1,1)$-form $\chi$. 
The (extended) K-energy  $\Kenergy: \mathcal E^1 \to (-\infty,\infty]$ is defined as follows:
\begin{equation}\label{eq: Kendfef}
\Kenergy(u):= 
\textup{Ent}(\omega^n,\omega^n_u)
+  \bar{S}\textup{AM}(u) 
-  n\textup{AM}_{\Ric \omega}(u),
\end{equation}
where $\textup{Ent}(\omega^n,\omega^n_u)$ is the entropy of the measure $\omega_u^n$ with respect to $\omega^n$: 
$$\textup{Ent}(\omega^n,\omega^n_u)=V^{-1}\int_X \log(\omega_u^n/\omega^n)\omega_u^n,$$ 
and $\overline{S}=\frac{1}{V}\int_X S_\omega \omega^n$ is the average scalar curvature, that can be seen to be independent of the choice of background metric. The $\chi$-twisted K-energy is defined as 
\begin{equation*}
	\Kenergy_{\chi} (u) = \Kenergy(u) + n\AM_{\chi}(u)-\frac{n}{V}\Big(\int_X \chi \wedge \omega^{n-1} \Big)\AM(u). 
\end{equation*}

When restricted to $\Homega$, the above formula for the K-energy was originally introduced by Chen--Tian \cite{ch1}, with a similar formula already appearing in \cite{t1}.  

The first order variation of $\Kenergy_\chi$ is given by the following formula:
\begin{equation}
	\label{eq: derivative of K chi}
	\langle D\Kenergy_\chi(u), \delta v\rangle = V^{-1} \int_X \delta v (\bar S_\chi-S_{\omega_u} + \Tr^{\omega_u}(\chi))\omega^n_u,
\end{equation}
where $\bar{S}_\chi = nV^{-1}\int_X (\Ric(\omega) - \chi) \wedge \omega^{n-1}$.
Hence, the critical points of $\Kenergy_{\chi}$ are the twisted csck potentials, as these satisfy $\bar S_{\chi}-S_{\omega_u} + \Tr^{\omega_u}(\chi)=0$.

In \cite[Proposition 5.26]{DR17} it is shown that this functional naturally extends to the $L^1$-Mabuchi completion of $\Homega$, which is just $\mathcal E^1$, and the extension is $d_1$-lsc (in \cite[Theorem 1.2]{bdl}). For more information on the metric spaces $(\Homega,d_p)$ we refer to \cite{da1,da2}, where these metric structures were introduced. In this note we will  focus on the case $p=1$.

 It follows from \cite[Theorem 4.7]{bdl} that the K-energy $\Kenergy$ as well as its twisted version $\Kenergy_{\chi}$  is convex along finite energy geodesics in $\Ec^1$.  We introduce the $\mathcal E^1$-minimizer set of $\Kenergy$ and $\Kenergy_{\chi}$:
\begin{flalign*}
\mathcal M^1 & := \{u \in \mathcal E^1 \cap \AM^{-1}(0)| \ \Kenergy(u) = \inf_{v \in \mathcal E^1}\Kenergy(v)\}, \\
\mathcal M^1_{\chi}& := 	 \{u \in \mathcal E^1 \cap \AM^{-1}(0)| \ \Kenergy_{\chi}(u) = \inf_{v \in \mathcal E^1}\Kenergy_{\chi}(v)\}. 
\end{flalign*}
\begin{remark}\label{rem: mimimizers}Given that $\Kenergy:\mathcal E^1 \to \Bbb R \cup \{ \infty\}$ is convex \cite[Theorem 1.2]{bdl}, it is straightforward to check that the minimizer set $\mathcal M^1$ (when non-empty) is totally geodesic with respect to the finite energy geodesics of $\mathcal E^1$. When $\chi >0$, using this, and the fact that $\AM_\chi$ is strictly convex on $\mathcal M^1$ (\cite[Theorem 4.12]{bdl}), if there exists   $v_\chi \in \mathcal M^1$ such that $\AM_\chi(v_\chi)=\inf_{v \in \mathcal M^1}\AM_\chi(v)$, then $v_\chi$ is unique. Also, \cite[Theorem 4.13]{bdl} shows that $\mathcal{M}^1_{\chi}$ contains at most one element. 
\end{remark}
 
\subsection{The action of automorphisms}
We let $G:= \AUT_0(X)$ denote the identity component of the  group of biholomorphisms of $X$. If $g\in G$ and $\varphi\in \Hnormalize$ we define $g.\varphi$ to be the unique element in $\Hnormalize$ such that 
\[
g^*(\omega_{\varphi}) = \omega +i\ddbar (g.\varphi).
\]
We refer the reader to \cite[Section 5.2, Lemmas 5.8-5.11]{DR17} for a detailed discussion on this action, which extends in a $d_1$-Lipschitz manner to $\Ec^1\cap \AM^{-1}(0)$.

\section{The finite energy continuity method}

Given a K\"ahler form $\chi$ cohomologous to $\omega$, in case $\mathcal M^1$ is non-empty, for any $\lambda >0$ we will show in Proposition \ref{prop: weak_continuity_prop} below  that it is possible to find a unique $v_\lambda \in \mathcal E^1 \cap \textup{AM}^{-1}(0)$  minimizing $\Kenergy_{\lambda \chi}$, giving rise to the \emph{finite energy continuity curve} $\lambda \to v_\lambda$. As we outline now, properties of this curve, especially in connection with the choice of twisting form $\chi$ play a vital role in the argument of Theorem \ref{thm: reguarity theorem}. 

\paragraph{Outline of the proof of Theorem \ref{thm: reguarity theorem}.} Before entering into the details we describe our argument briefly.  Assume that $\omega$ is a csck metric and let $v\in \mathcal{M}^1$.  Let $v_j \in \mathcal H_0$ such that $d_1(v,v_j) \to 0$ (such a sequence exists by \cite[Theorem 2]{da2}). For each $j$, setting $\chi_j:= \omega_{v_j}$, by Remark \ref{rem: mimimizers} we can find a unique $\varphi_j\in \mathcal{M}^1$ such that 
\[
\AM_{\chi_j}(\varphi_j) = \inf \{ \AM_{\chi_j}(u) \setdef u \in \mathcal{M}^1\}. 
\] 
Moreover, as shown in Proposition \ref{prop: weak_continuity_prop} below we have a uniform control of $\AubinI(\varphi_j,v)$ by $\AubinI(v_j,v)$. The latter goes to $0$ as $j\to +\infty$ (since $d_1(v,v_j) \to 0$), hence $\lim_{j\to +\infty} \AubinI(\varphi_j,v)=0$.   Generalizing the arguments of \cite[Theorem 4.4, Proposition 4.7]{bb}, we will show in Proposition \ref{prop: AMtwist_minimizer} that $\varphi_j$ is smooth and there exists an automorphism $g_j\in G$ such that $\varphi_j=g_j.0$.  In Lemma \ref{lem: autlimit} we use the reductiveness  of $G$ to show that $g_j$ converge smoothly to some $g\in G$. 

Putting everything together we arrive at $0=\lim_{j\to +\infty} \AubinI(g_j.0,v)=\AubinI(g.0,v)$. Since $\AubinI$ is non-degenerate, it then follows that $v=g.0$ is smooth csck, as desired. 

  \medskip

We now provide the details of the proof of Theorem \ref{thm: reguarity theorem}. 
\begin{proof}[Proof of Theorem \ref{thm: reguarity theorem}] Without loss of generality we can assume that $\omega$ is a csck metric and let $v\in \mathcal{M}^1$.  We want to show that $v = g.0$ for some $g \in G$. By \cite{da2} there exists $v_j \in \Hnormalize$ with $d_1(v_j,v) \to 0$. For notational convenience we set $\chi_j:= \omega+i\ddbar v_j$, which is a K\"ahler form.  

Fix $\lambda >0$ momentarily and consider the twisting forms $\lambda \chi_j$. Since $\omega$ is csck, the set $\mathcal{M}^1$ is non-empty. By Proposition \ref{prop: weak_continuity_prop} below the functional $\Kenergy_{\lambda \chi_j}$ admits a unique minimum $v^{\lambda}_j \in \mathcal E^1\cap \AM^{-1}(0)$ satisfying 
\[
\AubinI(v^{\lambda}_j,v_j) \leq n(n+1) \AubinI(v_j,v).
\] 
As $\AubinI$ satisfies the quasi-triangle inequality \eqref{eq: quasi triangle for Aubin I} we have in fact
\begin{flalign}\label{eq: lambda inequality}
\AubinI (v^\lambda_j,v) & \leq \frac{1}{c_n}  \left (\AubinI(v^{\lambda}_j, v_j) + \AubinI(v_j,v) \right )\nonumber \\ 
& \leq  \frac{1}{c_n} \left (n(n+1)\AubinI(v_j,v) + \AubinI(v_j,v) \right )\nonumber \\ 
&\leq \frac{n^2+n+1}{c_n} \AubinI(v_j,v). 
\end{flalign}
Fix $j\in \mathbb{N}$. It follows from Proposition \ref{prop: weak_continuity_prop} below that there exists $\varphi_j \in \mathcal{M}^1$ such that
\[
\lim_{\lambda \to 0^+}d_1(v^{\lambda}_j, \varphi_j)  =0 \ \ \textup{ and } \ \ \AM_{\chi_j} (\varphi_j) = \inf \{\AM_{\chi_j}(u) \setdef u \in \mathcal{M}^1\}. 
\]
Letting $\lambda \to 0^+$ in \eqref{eq: lambda inequality}, using the $d_1$-continuity of $I$ we arrive at 
\begin{equation}
	\label{eq: continuity path 1}
	\AubinI(\varphi_j,v) \leq \frac{n^2+n+1}{c_n} \AubinI(v_j,v). 
\end{equation}
It follows from Proposition \ref{prop: AMtwist_minimizer} below that, for each $j$, $\varphi_j$ is smooth and there exists $g_j\in G$ such that $\varphi_j=g_j.0$. Letting $j \to \infty$ in \eqref{eq: continuity path 1} we obtain $\AubinI(g_j.0,v) \to 0$, which by \cite[Proposition 2.3]{bbegz} is equivalent to $\|g_j.0-v\|_{L^1(X)} \to0$ and $\AM(g_j. 0) \to \AM(v)$. By \cite[Proposition 5.9]{da2} this is further equivalent to $d_1(g_j.0,v)\to 0$. Finally, by Lemma \ref{lem: autlimit} below there exists $g \in G$ such that $g.0=v$, finishing the proof.
\end{proof}

In the remaining part of this section we prove Proposition \ref{prop: weak_continuity_prop} and Proposition \ref{prop: AMtwist_minimizer}, that represent the main analytic tools in the proof of Theorem \ref{thm: reguarity theorem} above.

\begin{proposition}\label{prop: weak_continuity_prop}
Assume that $\mathcal M^1$ is nonempty and $u \in \Homega$. Then for any $\lambda > 0$, there exists a unique minimizer $v^\lambda \in\Ec^1\cap \AM^{-1}(0)$ of $\Kenergy_{\lambda\omega_u}$. The curve $[0,\infty) \ni \lambda \to v^\lambda \in \mathcal E^1\cap \AM^{-1}(0)$ is $d_1$-continuous, $d_1$-bounded with $v^0 = \lim_{\lambda \to 0}v^\lambda$ being the unique minimizer of $\AM_{\omega_u}$  on $\mathcal M^1$.  Additionally, for any $w \in \mathcal M^1, \ \lambda \geq 0$ we have
\begin{equation}\label{eq:Iestimate}
\AubinI (v^\lambda,u) \leq n(n+1) \AubinI (w,u).
\end{equation}
\end{proposition}
\begin{proof}First we show that the curve $\lambda \to v_\lambda$ described in the statement exists. Fixing $\lambda >0$, observe that $\Kenergy_{\lambda \omega_u}= \Kenergy + n\lambda \AM_{\omega_u}$ on $\Ec^1\cap \AM^{-1}(0)$. Let $u_j \in \mathcal E^1 \cap \AM^{-1}(0)$ be a minimizing sequence of $\Kenergy_{\lambda \omega_u}$. As $\mathcal M^1$ is nonempty, it follows that $\Kenergy$ is bounded from below. It follows from \eqref{eq: JIdoubleestimate} that $\AM_{\omega_u}(v) \geq \AM_{\omega_u}(v)$ for every $v\in \Ec^1\cap \AM^{-1}(0)$. Thus $\Kenergy_{\lambda \omega_u}$ is also bounded from below on $\Ec^1\cap \AM^{-1}(0)$.  These lower bounds together give that both $\AM_{\omega_u}(u_j)$ and $\Kenergy(u_j)$ are in fact uniformly bounded, $j \in \Bbb N$. 

As $[\omega_u]_{dR}=[\omega]_{dR}$, \cite[Proposition 5.5]{DR17} gives that $d_1(0,u_j)$ is also uniformly bounded. Using this and the uniform bound on $\Kenergy(u_j)$, we get that $\Ent(\omega^n,\omega^n_{u_j})\geq 0$ is also uniformly bounded above, hence we can apply \cite[Theorem 2.17]{bbegz} (see \cite[Theorem 5.6]{DR17} for an equivalent formulation that fits our context most). By this last compactness result, from $u_j$ we can extract a $d_1$-convergent subsequence, converging to $v^{\lambda}\in \mathcal E^1$. By the $d_1$-lower semi-continuity of $\Kenergy_{\lambda \omega_u}$ \cite[Theorem 1.2]{bdl} we get that $v^{\lambda}$ is a $\mathcal E^1$-minimizer of $\Kenergy_{\lambda \omega_u}$ and by \cite[Theorem 4.13]{bdl} this minimizer has to be unique.
\medskip

Now we prove \eqref{eq:Iestimate}. Let $w \in \mathcal M^1$ and $\lambda >0$. As $v^\lambda$ and $w$ minimize $\Kenergy_{\lambda\omega_u}$ and $\Kenergy$ respectively, we can write the following:
	\[
	\Kenergy(w) + \lambda n\AM_{\omega_u}(v^\lambda) \leq  \Kenergy(v^\lambda) +\lambda n \AM_{\omega_u}(v^\lambda)
		=\Kenergy_{\lambda \omega_u}(v^\lambda)
		\leq  \Kenergy_{\lambda \omega_u}(w)
		= \Kenergy(w) + \lambda n\AM_{\omega_u}(w),
		\]
hence $\AM_{\omega_u}(v^\lambda) \leq \AM_{\omega_u}(w)$. Subtracting  $\AM_{\omega_u}(u)$ from this and using \eqref{eq: JIdoubleestimate} we get \eqref{eq:Iestimate} for $\lambda > 0$. 
\medskip

We next argue that $\{v^\lambda\}_{\lambda >0}$ is $d_1$-bounded and relatively $d_1$-compact. It follows from \eqref{eq:Iestimate} and the quasi-triangle inequality \eqref{eq: quasi triangle for Aubin I} that  
\[
\AubinI(v^{\lambda}, 0) \leq \frac{1}{c_n} \left (\AubinI(v^{\lambda},u) + \AubinI(u,0) \right ) \leq \frac{1}{c_n} \left (n(n+1)\AubinI(u,w) + \AubinI(u,0) \right ).
\]
In particular $\AubinI(v^{\lambda},0)$ is uniformly bounded in $\lambda\in (0,1)$. Since $\AM(v^{\lambda})=0$ it follows from \eqref{eq: inequality for AM} that $\int_X v^{\lambda} \omega_{v_{\lambda}}^n \leq 0$. From this and the definition of the $\AubinI$ functional we see that $\int_X v^{\lambda} \omega^n$ is uniformly bounded. Thus \cite[Proposition 2.7]{GZ05} reveals that $\sup_X v^{\lambda}$ is uniformly bounded. 
 Hence \cite[Proposition 5.5]{DR17} implies that $d_1(0,v^\lambda)$ is uniformly bounded, $\lambda >0$. Observe that we have a trivial upper bound $\Kenergy_{\lambda \omega_u}(v^\lambda) \leq \Kenergy_{\lambda \omega_u}(0)=\Kenergy(0)$. Since all terms except the first in the expression of $\Kenergy_{\lambda \omega_u}(v^\lambda)$ from \eqref{eq: Kendfef} are bounded by $d_1(0,v^\lambda)$ it follows that  $\Ent(\omega^n,\omega^n_{v^\lambda}) \geq 0$ is also uniformly bounded from above, ultimately giving that $\{ v^\lambda\}_{\lambda >0}$ is relatively $d_1$-compact, again by \cite[Theorem 2.17]{bbegz}. 

\medskip

We claim now that $\lambda \to v^\lambda$ is $d_1$-continuous for $\lambda>0$. Indeed, assume that $\{\lambda_j\}_j$ converges to $\lambda>0$. As shown above, the sequence $v^{\lambda_j}$ is  relatively $d_1$-compact, hence it suffices to prove that any limit of this sequence coincides with $v^{\lambda}$. So, we can assume that $v^{\lambda_j}\to v$ and we will show that $v=v^{\lambda}$. For any $h\in \Ec^1$ we have
\begin{eqnarray*}
	\Kenergy_{\lambda_j\omega_u}(h) \geq \Kenergy_{\lambda_j\omega_u} (v^{\lambda_j}) = \Kenergy(v^{\lambda_j}) +\lambda_jn \AM_{\omega_u}(v^{\lambda_j}). 
\end{eqnarray*}
Letting $j\to +\infty$, we can use that $\Kenergy$ is $d_1$-lsc and $\AM_{\omega_u}$ is $d_1$-continuous to obtain that $\Kenergy_{\lambda\omega_u}(h)\geq \Kenergy_{\lambda \omega_u}(v)$. Uniqueness of minimizers of $\Kenergy_{\lambda\omega_u}$ \cite[Theorem 4.13]{bdl} now gives that $v^\lambda=v$, what we wanted to prove.

Finally, we focus on continuity at $\lambda =0$.  Using relative compactness of $\{ v^\lambda\}_{\lambda >0}$, we can find $\lambda_j \to 0$ and $v^0 \in \mathcal E^1$ such that $d_1(v^{\lambda_j},v^0) \to 0$. We will show that $v^0$ is independent of the choice of $\lambda_j$. By the joint lower semi-continuity of $(h,\lambda) \to \Kenergy_{\lambda \omega_u}(h)$ it follows that $v^0 \in \mathcal M^1$. Let $q \in \mathcal M^1$ be arbitrary. Then we have that 
$$
\Kenergy(q) \leq \Kenergy(v^{\lambda_j}) \ \ \textup{ and } \ \ \Kenergy_{\lambda_j \omega_u}(v^{\lambda_j}) \leq \Kenergy_{\lambda_j\omega_u}(q),
$$ 
implying that $n\lambda_j\AM_{\omega_u}(v^{\lambda_j}) \leq  n\lambda_j \AM_{\omega_u}(q)$, hence $\AM_{\omega_u}(v^{\lambda_j}) \leq \AM_{\omega_u}(q)$. After letting $j \to \infty$, it follows that $\AM_{\omega_u}(v^{0}) \leq \AM_{\omega_u}(q)$, hence $v^0$ is a  minimizer of $\AM_{\omega_u}$ on $\mathcal M^1\cap \AM^{-1}(0)$. Since $\omega_u$ is K\"ahler it follows from \cite[Theorem 4.12,Theorem 4.13]{bdl}  that $v^0$ is uniquely determined, finishing the proof.  
\end{proof}

As detailed in the next proposition, whose proof builds on the arguments of \cite{bb}, if a smooth csck metric exists, then the minimizer of $\AM_{\omega_u}$ on $\mathcal M^1$ can be given more specifically:

\begin{proposition}\label{prop: AMtwist_minimizer} 
Assume that $v$ is a csck  potential  and $u\in \Homega$. Then there exists  $g \in G$ such that
\[
\inf_{w \in \mathcal M^1}\AM_{\omega_u}(w)= \AM_{\omega_u}(g.v).
\]
\end{proposition}

Before elaborating the details of the proof, let us describe briefly the ideas, which build on the proof of the uniqueness of csck metrics in \cite[Theorem 4.4 and Proposition 4.7]{bb}.  

\paragraph{Outline of the proof of Proposition \ref{prop: AMtwist_minimizer}.} By changing the reference metric from $\omega$ to $\omega_u$, we can assume that $u=0$ and $\omega$ is csck.  As we will see, by reductiveness of $G$ we can find $g\in G$ such that $g.v$ minimizes $\AM_{\omega}$ on the orbit $G.v$. Let $v^0$ be the unique minimizer in $\mathcal{M}^1$ of $\AM_{\omega}$ (which exists by Remark \ref{rem: mimimizers}). The goal is to prove that $v^0=g.v$. 

Let $F_{\lambda}$ be the smooth differential of $E_{\lambda \omega}$. By the choice of $g$, following the arguments in \cite[Proposition 4.3 and Theorem 4.4]{bb}, we will get that there exists $h\in \Cc^{\infty}(X)$ such that for $\lambda \in \mathbb{R}$ small, 
\begin{equation}
	\label{eq: BB 2}
	F_{\lambda}(g.v + \lambda h)=O(\lambda^2).
\end{equation}
 Let $v^{\lambda}\in \Ec^1 \cap \AM^{-1}(0)$ be the unique minimizer of $\Kenergy_{\lambda \omega}$  on  $\Ec^1 \cap \AM^{-1}(0)$, $\lambda >0$.  As shown in Proposition \ref{prop: weak_continuity_prop} $d_1(v^{\lambda},v) \to 0$ as $\lambda\to 0$.  Let $t \to u^{\lambda}_t$ be the finite energy geodesic connecting $u^{\lambda}_0=v^{\lambda}$ and $u^{\lambda}_1= g.v + \lambda h$ (when $\lambda$ is small enough $g.v +\lambda h$ is  a K\"ahler potential). We want to prove that as $\lambda\to 0$ the  limiting geodesic $t \to u_t$ (which is known to exist by the endpoint stability of finite energy geodesics \cite[Proposition 4.3]{bdl}) is trivial, i.e., $v^0 = u_0 = u_1 = g.v$. Using \eqref{eq: BB 2} we will first show that  
 \[
 \left(\frac{d}{dt}\Big|_{t=1^-} - \frac{d}{dt}\Big|_{t=0^+}\right)  \AM_{\omega} (u^{\lambda}_t) \leq C |\lambda|,
 \]
 for a constant $C>0$, independent of $\lambda.$
 
 Using only the convexity of $t \to \AM_{\omega}(u^\lambda_t)$ \cite[Theorem 4.12]{bdl} (after possibly increasing $C$) this implies in an elementary manner that
$$
0\leq t \AM_{\omega} (u^\lambda_1)+ (1-t)\AM_{\omega} (u^\lambda_0)-\AM_{\omega}(u^\lambda_t) \leq t (1-t) C|\lambda| , \ t\in [0,1].
$$
Letting $\lambda \to 0$  we get that $t \to \AM_{\omega}(u_t)$ is linear. Since potentials along $t \to u_t$ have finite entropy, the last statement of \cite[Theorem 4.12]{bdl} implies that $t \to u_t$ has to be a constant curve, i.e., if $w=u_0=u_1=g.v$, finishing the argument.
 
The precise argument will rely heavily on \cite[Proposition 4.3, Theorem 4.4, Proposition 4.7]{bb}. Compared to \cite{bb}, the first main difference in our analysis is that one end point ($t=1$) on our finite energy  geodesic $t \to u^\lambda_t$ is smooth (hence we can handle the derivatives) while the other end point ($t=0$) is apriori singular (in $\Ec^1$).
The second main difference is our use of the twisted K-energy. In \cite[Proposition 4.3, Theorem 4.4, Proposition 4.7]{bb} the perturbation term is given by the $J$ type  functional. For our argument to work, we need strict convexity of $\AM_{\omega}$ along our finite energy geodesic $t \to u_t$, which was proved in \cite[Theorem 4.12 and Theorem 4.13]{bdl}. 
 \medskip
 
 We now explain the  proof of Proposition \ref{prop: AMtwist_minimizer} in detail.

\begin{proof}[Proof of Proposition \ref{prop: AMtwist_minimizer}]
By changing the reference metric from $\omega$ to $\omega_u$, we can assume that $u=0$ and $\omega$ is csck. As a csck metric exists, the group $G$ is reductive, hence there exists $g\in G$ such that 
$$\AM_{\omega}(g.v)=\min \{\AM_{\omega}(f.v) \setdef f\in G\}$$ 
This is indeed well known and can be seen from the fact that $\AM_{\omega}$ is equivalent with the growth of the $d_1$-metric \cite[Proposition 5.5]{DR17}, and the Lie algebra of $G$ has a very specific decomposition (for details see for example Section 6 of \cite{DR17}, especially \cite[Proposition 6.2, Proposition 6.9]{DR17}).

We denote $\tilde v^0=g.v \in \Hnormalize$ and  let $v^0\in \mathcal{M}^1$ be the unique minimizer of $\AM_{\omega}$ on $\mathcal M^1$, known to exist by the previous proposition. We are done if we can show that $v^0=\tilde v^0$. For each $\lambda\in (0,1/2]$ let $v^{\lambda}$ be the unique $\Ec^1$-minimizer of $\Kenergy_{\lambda\omega}$ on $\Ec^1 \cap \AM^{-1}(0)$. Then by Proposition \ref{prop: weak_continuity_prop} above, $d_1(v^\lambda,v^0)\to 0$ as $\lambda \to 0$. Let $F_\lambda$ and $W$ denote the differential of $\Kenergy_{\lambda\omega}|_{\Homega}$ and $n\AM_{\omega}|_{\Homega}$ respectively. 

	Given the specific choice of $\tilde v_0$, by the same argument as in the proof of \cite[Theorem 4.4, Proposition 4.7]{bb} we can find $h\in \Cc^{\infty}(X)$  such that 
	$$
	D_{h} F_0|_{\tilde v^0} = -W(\tilde v^0).
	$$
Again going back to the arguments in \cite[Theorem 4.4, Proposition 4.7]{bb}, for small enough $\lambda \geq 0$ this identity implies
	$$
	|F_\lambda(\tilde v^0+\lambda h).w |\leq C\lambda^2\sup_X |w|, \ \forall w\in \Cc(X).
	$$
Given this last estimate and the explicit formula for $F_\lambda$ \eqref{eq: derivative of K chi} we can further write: 
\begin{equation*}
	F_{\lambda}(\tilde v^0+\lambda h). w = \int_X w f_\lambda \omega_{\tilde v^0 + \lambda h}^n,
\end{equation*}
where  $f_\lambda=\overline{S}_{\lambda\omega}-S_{\omega_{\tilde v^0+ \lambda h}} + \lambda \Tr^{\omega_{\tilde v^0+ \lambda h}}(\omega) \in \Cc^\infty(X)$ satisfies $f_\lambda=O(\lambda^2)$. 

Let $[0,1] \ni t \to u_t^\lambda \in \mathcal E^1$ be the finite energy geodesic connecting $u_0^\lambda:= v^\lambda \in \mathcal E^1$ with  $u^\lambda_1:=\tilde v^0+\lambda h \in \Homega$ for $\lambda$ small enough.
By Lemma \ref{lem:deriv av mab along singular} below we can write:
	$$
	\frac{d}{dt}\Big|_{t=1^-} \Kenergy_{\lambda\omega} (u^\lambda_t) \leq \int_X{\dot u}_1^\lambda f_\lambda \omega_{\tilde v^0 + \lambda h}^n.
	$$
Proposition \ref{prop: weak_continuity_prop} and the fact that $u^\lambda_1=\tilde v^0 + \lambda h$ is smooth gives that the quantities $d_1(0,u^\lambda_1)$ and $d_1(0,u^\lambda_0)$ are uniformly bounded. Consequently, by Lemma \ref{lem: geod_dist_limit}(ii) below we obtain the following estimate:
\begin{flalign*}
\int_X|{\dot u}_1^\lambda | \omega_{u^\lambda_1}^n =d_1(u^\lambda_1,u^\lambda_0) \leq d_1(0,u^\lambda_0) + d_1(0,u^\lambda_1)\leq C.
\end{flalign*} 
Since $f_\lambda = O(\lambda^2)$ we can ultimately write
$$
\frac{d}{dt}\Big|_{t=1^-} \Kenergy_{\lambda\omega} (u_t^\lambda) \leq O(\lambda^{2}). 
$$
Recall that $u^\lambda_0=v^\lambda$ is the unique $\Ec^1$-minimizer of the convex functional $\Kenergy_{\lambda\omega}$, thus $\frac{d}{dt}\big|_{t=1^-}\Kenergy_{\lambda\omega}(u^\lambda_t) \geq \frac{d}{dt}\big|_{t=0^+}\Kenergy_{\lambda\omega}(u^\lambda_t) \geq 0$. Consequently, as both $t\to \AM_\omega(u^\lambda_t)$ and $t \to \Kenergy(u^\lambda_t)$ are convex, we obtain the following sequence of estimates
$$
	0\leq n\lambda \Big(\frac{d}{dt}\Big|_{t=1^-} -\frac{d}{dt}\Big|_{t=0^+}\Big) \AM_{\omega}(u_t^\lambda) \leq \Big(\frac{d}{dt}\Big|_{t=1^-} -\frac{d}{dt}\Big|_{t=0^+} \Big)\Kenergy_{\lambda\omega}(u^\lambda_t) \leq O(\lambda^{2}).  
$$
Using convexity of $t\to \AM_\omega(u^\lambda_t)$ again, this last estimate gives
$$
0\leq t \AM_{\omega} (u^\lambda_1)+ (1-t)\AM_{\omega} (u^\lambda_0)-\AM_{\omega}(u^\lambda_t) \leq t (1-t) O(\lambda) , \ t\in (0,1).
$$
Letting $\lambda \to 0$,  using the endpoint stability of finite energy geodesic segments \cite[Proposition 4.3]{bdl} and the $d_1$-continuity of $\AM_\omega$ \cite[Lemma 5.23]{DR17}, we obtain that  $t \to \AM_{\omega}(u_t)$ is linear along the finite energy geodesic $[0,1] \ni t \to u_t \in \mathcal E^1$ connecting $v^0$ to $\tilde v^0$. By \cite[Theorem 4.12]{bdl} this implies that $v^0=\tilde v^0$, what we desired to prove. 
\end{proof}

As promised in the above argument, we provide the following lemma, which generalizes \cite[Lemma 3.5]{bb} to finite energy geodesics with one smooth endpoint:

\begin{lemma}
\label{lem:deriv av mab along singular}Given $u_{1}\in\mathcal{E}^{1}$ and $u_{0}\in\mathcal{H}_\omega$ let $[0,1] \ni t \to u_{t} \in \mathcal E^1$ be the finite energy geodesic connecting $u_0,u_1$ and $\chi$ is a smooth closed and positive $(1,1)$-form. Then 

\[
\lim_{t\rightarrow0^{+}}\frac{\Kenergy_\chi(u_{t})-\Kenergy_\chi(u_{0})}{t}\geq\int_{X}(\bar S_\chi-S_{\omega_{u_{0}}} + \textup{Tr}^{\omega_{u_0}}\chi)\dot u_0\omega_{u_{0}}^{n},
\]
where $\bar S_\chi = {n}{V}^{-1}\int_X (\Ric \omega - \chi) \wedge \omega^{n-1}$.
\end{lemma}

\begin{proof} Using Theorem \cite[Theorem 1.2]{bdl} it is enough to show that
\begin{equation}\label{eq: toprove}
\frac{\Kenergy_\chi(u_{t})-\Kenergy_\chi(u_{0})}{t}\geq\int_{X}(\bar S_\chi -S_{\omega_{u_{0}}} + \textup{Tr}^{\omega_{u_0}}\chi)\dot u_0\omega_{u_{0}}^{n}, \ t \in [0,1].
\end{equation}
Fix $t\in [0,1]$. By \cite[Theorem 1.2]{bdl} there exists  $u^k_t \in \Homega$ such that $d_1(u^k_t,u_t) \to 0$ and $\Kenergy_\chi(u_t^k) \to \Kenergy_\chi(u_t)$. Let $[0,t] \ni l \to v^{k}_l \in \mathcal E^1$ be the weak $C^{1\bar1}$ geodesic connecting $u^k_0:=u_0$ with $u^k_t$. By \cite[Lemma 3.5]{bb} we can write:
$$\frac{\Kenergy_\chi(u^k_{t})-\Kenergy_\chi(u_{0})}{t}\geq\int_{X}(\bar S_\chi-S_{\omega_{u_{0}}} + \textup{Tr}^{\omega_{u_0}}\chi)\dot v_0^{k}\omega_{u_{0}}^{n}.$$
By the next lemma, after perhaps passing to a subsequence, we can apply the dominated convergence theorem on the right hand side and obtain \eqref{eq: toprove}. 
\end{proof}

\begin{lemma}\label{lem: geod_dist_limit}Suppose $u^j_1,u_1 \in \mathcal E^1$ satisfies $d_1(u^j_1,u_1) \to 0$ and $u_0 \in \Homega$.  Let $[0,1] \ni t \to u_t,u^j_t \in \mathcal E^1$ be the finite energy geodesics connecting  $u_0,u_1$ and $u_0,u^j_1$ respectively. Then the following hold:\\
\noindent (i) There exists $f \in L^1({\omega_{u_0}^n}), \ {j_k} \to \infty$ such that $|\dot u^{j_k}_0| \leq f$ and $\dot u^{j_k}_0 \to \dot u_0$ a.e. \\
\noindent (ii) $d_1(u_0,u_1)=\int_X |\dot u_0| \omega_{u_0}^n.$
\end{lemma}

\begin{proof} If $[0,1] \ni t \to v_t \in \mathcal E^1$ is an arbitrary finite energy geodesic, we observe that $t\to v_t + t \alpha  + (1-t)\beta$ is the finite energy geodesic connecting $v_0 +\beta$ and $v_1+\alpha$, $\alpha,\beta \in \Bbb R$. Using this observation, to establish (i) we can assume without loss of generality that 
\begin{equation}\label{eq: monotone_assumption}
u_0-1 \geq u_1,u^j_1. 
\end{equation}
We first show that (ii) holds in this particular case:
\begin{equation}\label{eq: distgeodeq} 
d_1(u_0,u_1) = \int_X -\dot u_0 \omega_{u_0}^n.
\end{equation}
Indeed, let $\tilde u^j_1 \in \Homega$ be a sequence decreasing to $u_1$ with $\tilde u^j_1 \leq u_0$. By \cite[Theorem 1]{da2}, since $t \to \tilde u^j_t$ is monotone decreasing we have $d_1(u_0,\tilde u^j_1)=\int_X -\dot {\tilde u}^j_0 \omega_{u_0}^n$, where $t \to \tilde u_t^j$ is the weak $C^{1\bar 1}$ geodesic connecting $u_0,\tilde u^j_1$, which is decreasing in $t$. As $\tilde u_0^j = u_0$ and $u^j_t\searrow u_t$, \eqref{eq: distgeodeq} follows from the monotone convergence theorem.

By \cite[Proposition 2.6]{bdl} there exists $j_k \to \infty$, $v^{j_k}_1 \in \mathcal E^1$ increasing and $w^{j_k}_1 \in \mathcal E^1$ decreasing such that $v^{j_k}_1 \leq u^{j_k}_1 \leq w^{j_k}_1 \leq u_0$ and  $d_1(u_1,v^{j_k}_1),d_1(u_1,w^{j_k}_1) \to 0$. Let $[0,1] \ni t \to v^{j_k}_t,w^{j_k}_t \in \mathcal E^1$ be the finite energy geodesics connecting $u_0, v^{j_k}_1$ and $u_0,w^{j_k}_1$ respectively. By the comparison principle for finite energy geodesic segments we ultimately get $\dot v^{j_k}_0 \leq \dot u^{j_k}_0 \leq \dot w^{j_k}_0 \leq 0$. We claim that for the monotone limits $\dot v_0 := \lim_k \dot v^{j_k}_0$ and $\dot w_0:=\lim_k \dot w^{j_k}_0$ we have $\dot v_0=\dot w_0=\dot u_0$ a.e., with this showing that $\dot u^{j_k}_0 \to \dot u_0$ a.e. as $k \to \infty$. Indeed, by \eqref{eq: distgeodeq} we have $d_1(u_0,w^{j_k}_1)= \int_X - \dot w^{j_k}_0 \omega_{u_0}^n$ and $d_1(u_0,v^{j_k}_1)= \int_X - \dot v^{j_k}_0 \omega_{u_0}^n$. Applying the monotone/dominated convergence theorems, we can write $d_1(u_0,u_1)=\int_X - \dot v_0 \omega_{u_0}^n=\int_X - \dot w_0 \omega_{u_0}^n.$
Since $\dot v_0 \leq \dot u_0 \leq \dot w_0$, it follows that $\dot v_0=\dot w_0=\dot u_0$ a.e. with respect to $\omega_{u_0}^n$, as we claimed.

Finally, as $\dot v^{j_1}_0 \leq \dot u^{j_k}_0 \leq 0$, using \eqref{eq: distgeodeq} we conclude that the function $f =|\dot v^{j_1}_0|$ satisfies the requirements of (i).

To argue (ii), let $\tilde u^j_1 \in \Homega$ be the same decreasing approximating sequence from the beginning of the proof. As \eqref{eq: monotone_assumption} may not hold, by \cite[Theorem 1]{da2} we only have $d_1(u_0,\tilde u^j_1)=\int_X |\dot {\tilde u}^j_0|\omega_{u_0}^n$. By (i), after perhaps passing to a subsequence, we can use the dominated convergence theorem to finish the proof.
\end{proof}

\begin{remark}In the proof of Lemma \ref{lem:deriv av mab along singular} above, by using the Ricci flow techniques of \cite{gz2,DL}, it is even possible to approximate $u_t$ by a decreasing sequence of smooth potentials with convergent K-energy.
\end{remark}

\begin{remark} Propositions \ref{prop: weak_continuity_prop}, \ref{prop: AMtwist_minimizer} together imply that whenever a csck potential $v \in \Hnormalize$ exists, then every ``finite energy continuity path" $[0,\infty) \ni \lambda \to v^\lambda \in \mathcal E^1\cap \AM^{-1}(0)$ $d_1$-converges to $g.v$ for some $g \in G$, with the crucial uniform estimate \eqref{eq:Iestimate}. Though we will not need it in this work, it is worth noting  that (using the implicit function theorem and additional estimates) in \cite[Theorem 1.1]{cpz} it is shown that $v^\lambda \in \Homega$ for small enough $\lambda$, and in fact $v^\lambda \to_{C^\infty} g.v$.
\end{remark}

Finally, we address the last auxiliary result in the proof of Theorem \ref{thm: reguarity theorem}:

\begin{lemma}\label{lem: autlimit} Suppose $\Homega$ contains a csck potential. If $u \in \Hnormalize$ and $g_j \in G$ are such that $d_1(g_j. u,h) \to 0$ for some $h \in \mathcal E^1$ as $j \to \infty$, then there exists $g \in G$ such that $g.u=h$.
\end{lemma}

\begin{proof} Let $v \in \Hnormalize$ be a csck potential. By \cite[Propositions 6.2 and 6.9]{DR17} there exists $k_j \in \textup{Isom}_0(X,\omega_v)$ and a Hamiltonian vector field $X_j \in \textup{isom}(X,\omega_v)$ such that $g_j = k_j \textup{exp}_I{JX_j}$. It is clear from the definition of the action of $G$ on the level of potentials that $k_j.v=v$. Thus we can write
\begin{flalign*}
d_1(v,g_j.u)&=d_1(v,k_j \textup{exp}_I(JX_j).u)=d_1(k_j^{-1}v, \textup{exp}_I(JX_j).u)=d_1(v, \textup{exp}_I(JX_j).u)\\
&=d_1(\textup{exp}_I(-JX_j)v,u) \geq d_1(\textup{exp}_I(-JX_j)v,v)-d_1(v,u),
\end{flalign*}
giving that $d_1(\textup{exp}_I(-JX_j)v,v)$ is bounded independently of $j$. As shown in see \cite[Section 7.1]{DR17} the curve $[0,\infty) \ni t \to \textup{exp}_I(-tJX_j).v \in \Homega\cap \AM^{-1}(0)$ is a $d_1$-geodesic ray, hence $\| X_j\|$ has to be uniformly bounded in $\textup{isom}(X,\omega_v)$. By compactness, after possibly relabeling the sequences, we can choose $X_{\infty} \in \textup{isom}(X,\omega_v)$ and $k \in \textup{Isom}_0(X,\omega_v)$ such that $k_j \to k$ and $X_j \to X_{\infty}$ smoothly, hence also $g_j =k_j \textup{exp}_I(JX_j) \to g: = k \textup{exp}_I(JX_{\infty})$ smoothly. In particular this implies $d_1(g_j.u,g.u) \to 0$, hence $g.u=h$ by the non-degeneracy of $d_1$.
\end{proof}

\section{K-polystability as a consequence of properness}

As mentioned in the introduction, in our proof of Theorem \ref{thm: Kpolystable} we will use a geometric reasoning involving geodesic rays.  G. Tian has informed us that the original ideas from \cite{t2} can also be generalized to the case when $G$ is non-trivial and the central fiber of a test configuration is non-normal (
via a  Moser iteration argument).

Before getting into exact details, first we outline our argument. In case a csck metric exists, from Theorem \ref{thm: DRthm} it trivially follows that $\Kenergy$ is bounded from below, hence $(X,\omega)$ is K-semistable. To prove K-stability, one has to show that test configurations $(\mathcal X,\mathcal L,\pi,\rho)$ with zero Donaldson-Futaki invariant $(DF(\mathcal X,\mathcal L)=0)$ are product test-configurations induced by a holomorphic vector field of $(X,L)$. By an estimate of the first named author (recalled \eqref{eq: geodtestconf_ineq}) and properness of the K-energy, we obtain that for test configurations satisfying $DF(\mathcal X,\mathcal L)=0$, the associated Phong-Sturm ray $t \to \phi_t$  satisfies a vital estimate when composed with $J_\omega$ (see \eqref{eq: J_0_geod_est}). Using this estimate we show that $t \to \phi_t$ is induced by the action of a Hamilton vector field of $X$ (Lemma \ref{lem: JG_bounded_ray}). Lastly, by a result of the first author (recalled in Proposition \ref{prop: Berman_prop}) it follows that $(\mathcal X,\mathcal L,\pi,\rho)$ is induced by a vector field of $(X,L)$. 

To give the precise argument, let us first fix some terminology. Let $L \to X$ be an ample line bundle over a K\"ahler manifold $(X,\omega)$ such that $c_1(L)=[\omega]$.  A test configuration $(\mathcal L,\mathcal X,\pi, \rho)$ for $(X,L)$ consists of a scheme $\mathcal X$ with a $\Bbb C^*$-equivariant flat surjective morphism $\pi: \mathcal X \to \Bbb C$ and a relatively ample line bundle $\mathcal L \to \mathcal X$ with a $\Bbb C^*$-action $\tau \to \rho_\tau$ on $\mathcal L$ such that $(X_1,\mathcal L|_{X_1}) = (X,kL)$ for some $k>1$. Without loss of generality we can assume that $k=1$, by treating  $\mathcal L$ as a $\Bbb Q$-line bundle. Following the findings of \cite{lx}, we will always assume that $\mathcal X$ is normal, which automatically makes the projection $\pi$ flat.

Given any test configuration $(\mathcal L,\mathcal X,\pi, \rho)$, after raising $\mathcal L$ to a sufficiently high power, it is possible to find an equivariant embedding into $\Bbb C \Bbb P^N \times \Bbb C$, such that $\mathcal L$ becomes the pullback of the relative $\mathcal O(1)$-hyperplane bundle (see  \cite{do1,thomas,ps1}). 
This automatically allows to fix a semi-positive smooth ``background" metric $h$ on $\mathcal L$, that is positive on every $X_\tau$ slice and is $S^1$-invariant. For the restrictions we introduce the notation $h_\tau = h|_{X_\tau}$, $\tau \in \Bbb C$.%

Any other positive metric $\tilde h$ on $\mathcal L$ can be uniquely represented by a potential $u^{\tilde h, \mathcal X} \in \textup{PSH}(\mathcal X, \Theta(h))$ using the identification 
$$\tilde h = h e^{-u^{\tilde h, \mathcal X}}.$$ 
Additionally, one can associate to $\tilde h$ another potential $u^{\tilde h, \Bbb C^*} \in \textup{PSH}(\Bbb C^*\times X,\textup{pr}_2^*\Theta(h_1))$ using the identification 
\begin{equation}\label{eq: testconf_to_subgeod}
\rho_{\tau}^*\tilde h\big |_{X_1} = h_1 e^{-u^{\tilde h,\Bbb C^*}_\tau}, \ \tau \in \Bbb C^*.
\end{equation}

By analyzing the action of $\rho$ restricted to global sections of $\mathcal L^r, \ r \geq 1$ on $X_0$, we can associate to $(\mathcal X,\mathcal L, \pi, \rho)$ the Donalson--Futaki invariant $DF(\mathcal X,\mathcal L)$. For details we refer to \cite{sz2,thomas}. We say that $(X,L)$ is K-polystable if for any test configuration $(\mathcal X,\mathcal L, \pi, \rho)$ we have $DF(\mathcal X,\mathcal L)\geq 0$, with $DF(\mathcal X,\mathcal L)= 0$ if and only if $\mathcal X$ is a product.

Let us fix $\phi \in \textup{PSH}(X,\Theta(h_1))$. According to Phong--Sturm \cite{ps1,ps2} (see also \cite[Section 2.4]{b}), to $(\mathcal X,\mathcal L,\pi,\rho)$ one can also associate a bounded geodesic ray $[0,\infty) \ni t \to \phi_t \in \textup{PSH}(X,\Theta(h_1)) \cap L^\infty$ (with $\phi_0 =\phi$) by first constructing a metric $\tilde h := he^{-\phi^{\mathcal X}}$ on $\mathcal L$, using the following upper envelope:
$$\phi^{\mathcal X} = \sup \{v \in \textup{PSH}(\mathcal X|_{\overline{\Delta}},\Theta(h)), v_\tau \leq \rho_{\tau^{-1}}^* \phi, \ |\tau|=1 \}.$$
The envelope $\phi^{\mathcal X}$ is seen to be $S^1$-invariant, and one can introduce $\phi_t = \phi^{\tilde h, \Bbb C^*}_{e^{-t/2}} \in \textup{PSH}(X,\omega) \cap L^\infty(X)$ for any $t \in [0,\infty)$. As argued in \cite{ps1,ps2}, this last curve $t \to \phi_t$ is indeed a weak $C^{1\bar1}$-geodesic ray. In general, $t \to \phi_t$ is not normalized, i.e., $\AM(\phi_t)$ is not identically zero (as this depends on the  $\Bbb C^*$-action). As follows from the proof of \cite[Proposition 2.7]{b} (see specifically the  argument that gives (2.16),(2.17)), there exists $C:=C(\phi,\mathcal L,\mathcal X,h)>0$ such that 
\begin{equation}\label{eq: geodtestconf_ineq}
h e^{-C} \leq h e^{-\phi^{\mathcal X}} \leq h e^C.
\end{equation}

\begin{proof}[Proof of Theorem \ref{thm: Kpolystable}]  Let $(\mathcal X, \mathcal L,\pi,\rho)$ be a test configuration equivariantly embedded into $\Bbb C \Bbb P^N \times \Bbb C$ with a $\Bbb C^*$-action $\Bbb C^* \ni \tau \to \rho_\tau \in GL(N+1,\Bbb C)$. By possibly composing $\rho_\tau$ with an inner automorphism, we can assume that the $S^1$-invariant background metric is just the restriction of the relative Fubini-Study metric $h^{FS}$ on $\mathcal O(1) \to \Bbb C \Bbb P^N \times \Bbb C$. For the background K\"ahler metric on $X$ we choose $\omega:= \Theta(h_1^{FS})$.  

We will prove that $DF(\mathcal X,\mathcal L) \geq 0$ with $DF(\mathcal X,\mathcal L) = 0$ if and only if $(\mathcal X, \mathcal L,\pi,\rho)$ is a product test configuration. 

We will be relying on the following formula relating the Donaldson-Futaki invariant to the asymptotics of the K-energy \cite{pt,prs,ti4,bhj,SD}:

\begin{equation}\label{eq: Tian_BHJ}
\Kenergy(u^{h^{FS},\Bbb C^*}_{\tau})  = -(DF(\mathcal X,\mathcal L) - a(\mathcal X,\mathcal L))\log |\tau|^2 + O(1), \ \tau \in \Bbb C^*,
\end{equation}
where $ a(\mathcal X,\mathcal L)\geq 0$, and $ a(\mathcal X,\mathcal L)= 0$ precisely when the central fiber $X_0$ is reduced (recall the notation introduced in \eqref{eq: testconf_to_subgeod} above). From Theorem \ref{thm: TianCor1} it follows that $\Kenergy$ is bounded from below, hence $DF(\mathcal X,\mathcal L) - a(\mathcal X,\mathcal L) \geq 0$, giving that $DF(\mathcal X,\mathcal L)\geq 0$.

Now assume that $DF(\mathcal X,\mathcal L)=0$. To finish the proof we will argue that $(\mathcal X, \mathcal L,\pi,\rho)$ is a product test configuration. Let $\phi \in \Hnormalize$ be a csck potential (recall that $\Theta(h_1)=\omega$ by choice) and let $[0,\infty) \ni t \to \phi_t \in \mathcal E^1$ be the associated $C^{1\bar 1}$-geodesic ray with $\phi_0 =\phi$.

First notice that \eqref{eq: K_is_J_Gproper} and Theorem \ref{thm: TianCor1} gives 
\[
\inf_{g \in G}J_{\omega}(g.(u^{h^{FS},\Bbb C^*}_{e^{-t/2}}-\AM(u^{h^{FS},\Bbb C^*}_{e^{-t/2}}))) < C'.
\]
Pulling back the estimates of \eqref{eq: geodtestconf_ineq} by $\rho_\tau$ and taking the $\log$, we immediately obtain  
\[
u^{h^{FS},\Bbb C^*}_{e^{-t/2}}-C \leq \phi_t \leq u^{h^{FS},\Bbb C^*}_{e^{-t/2}}+C.
\]
Using monotonicity of $\AM$, this further implies that 
\[
\phi_t - \AM(\phi_t) - 2C\leq u^{h^{FS},\Bbb C^*}_{e^{-t/2}}-\AM(u^{h^{FS},\Bbb C^*}_{e^{-t/2}}) \leq \phi_t - \AM(\phi_t) + 2C.
\]
Putting the above facts together and using also the monotonicity of $\AM$, after possibly increasing $C'$, we arrive at:
\begin{equation}\label{eq: J_0_geod_est}
\inf_{g \in G}J_{\omega}(g.(\phi_t-\AM(\phi_t))) < C'.
\end{equation}
Given that $\phi_0$ is a csck potential, Lemma \ref{lem: JG_bounded_ray} below implies that the normalized ray $t \to \phi_t-\AM(\phi_t)$ is induced by $t \to \textup{exp}_{I}(tJV)$, where $V$ is a real holomorphic Hamiltonian Killing field of $(X,J,\omega)$. By Lemma  \ref{lem: vectorlift}, it is even possible to find a lift $\tilde V$  to $L \to X$ such that  
$$\textup{exp}_I(tJ\tilde V)^* h_1^{FS}e^{-\phi_0} =h_1^{FS}e^{-\phi_t}.$$
Since $DF(\mathcal X,\mathcal L)=0$, \eqref{eq: Tian_BHJ} gives that $a(\mathcal X,\mathcal L)=0$, hence $X_0$ is reduced. Consequently, we can apply Proposition \ref{prop: Berman_prop} below to conclude that $\mathcal X$ is isomorphic to $X \times \Bbb C$, i.e., $\mathcal X$ is a product test configuration.
\end{proof}

We next turn to the statements and proofs of the auxiliary results invoked above.

\begin{lemma}\label{lem: JG_bounded_ray} Suppose $(X,\omega)$ is a K\"ahler manifold. Let $u_0 \in \Hnormalize$ be a csck potential and a finite energy geodesic ray $[0,\infty) \ni t \to u_t \in \mathcal E^1 \cap \AM^{-1}(0)$ emanating from $u_0$. If there exists $C>0$ such that
$$\inf_{g \in G}J_\omega(g.u_t) < C, \ t \in [0,\infty),$$
then there exists a real holomorphic Hamiltonian vector field $V \in \textup{isom}(X,\omega_{u_0})$ such that $u_t = \textup{exp}_I(tJV). u_0$, where $ t \to \textup{exp}_I(tJV)$ is the flow of $JV$.
\end{lemma}
\begin{proof}Let $g_k \in G$ such that $J_\omega(g_k.u_k) < C$. As $u_0$ is a csck potential there exists $h_k \in \textup{Isom}_0(X,\omega_{u_0})$ and a Hamiltonian vector field $V_k \in \textup{isom}(X,\omega_{u_0})$ such that $g_k = h_k \textup{exp}_I(-JV_k)$ (see \cite[Propositions 6.2 and 6.9]{DR17}).  As the growth of the $J+\o$ functional is the same as that of the $d_1$ metric \cite[Proposition 5.5]{DR17}, and $G$ acts by $d_1$-isometries on $\mathcal E^1 \cap \AM^{-1}(0)$ \cite[Lemma 5.9]{DR17}, by possibly increasing the constant $C$ we can write:
\begin{equation}\label{eq: geodineqineq}
C > d_1(u_0,g_k.u_k)=d_1(g_k^{-1}u_0,u_k)=d_1(\textup{exp}(JV_k).u_0,u_k).
\end{equation}
We can assume without loss of generality that $t \to u_t$ has unit $d_1$-speed, i.e., $d_1(u_0,u_t)=t$. Using the above inequality, the triangle inequality gives the following double estimate: 
$$k - C \leq d_1(u_0,\textup{exp}_I(JV_k).u_0) \leq k +C.$$
The analytic expression of $\textup{exp}_I(JV_k).u_0$ (see \cite[Lemma 5.8]{DR17}) implies that in fact $1/D \leq \| JV_k/k\| \leq D$ for some $D>1$. As the space of holomorphic Hamiltonian Killing fields of $(X,\omega_{u_0},J)$ is finite dimensional, it follows that there exists a nonzero Killing field $V$ such that $V_{k_j}/{k_j} \to V$ for some $k_j \to \infty$. 

Let us introduce the smooth $d_1$-geodesic segments
\[
[0,k] \ni t \to u^k_t = \textup{exp}_I\Big(t \frac{JV_k}{k}\Big).u_0 \in \Hnormalize.
\]
By \cite[Proposition 5.1]{bdl} the function $t \to d_1(u^k_t,u_t)$ is convex, hence \eqref{eq: geodineqineq} gives that $d_1(u^k_t,u_t) \leq Ct/k, \ t \in [0,k]$. This implies that for fixed $t$ we have $d_1(u^k_t,u_t) \to 0$. But examining convergence in the expressions defining $u^{k_j}_t = \textup{exp}_I(tJV_{k_j}/{k_j}).u_0$ we conclude that $u^{k_j}_t \to \textup{exp}_I(tJV).u_0$ smoothly, ultimately giving $u_t =\textup{exp}_I(tJV).u_0$.
\end{proof}

In case the K\"ahler class is integral, we have the following addendum to the previous lemma:

\begin{lemma} \label{lem: vectorlift} Suppose $(L,h) \to X$ is a hermitian line bundle with $\omega:=\Theta(h)>0$. Let $\phi_0 \in \Homega$, a real holomorphic Hamiltonian vector field $V \in \textup{isom}(X,J,\omega_{u_0})$, and $[0,\infty) \ni t \to \phi_t \in \mathcal E^1$ a geodesic ray. If the ``normalization" of $t \to \phi_t$ is induced by $V$, i.e., $\phi_t-\AM(\phi_t) = \textup{exp}_I(tJV).(\phi_0-\AM(\phi_0))$, then it is possible to find a lift $\tilde V$ of $V$ to the line bundle $L \to X$ such that $\textup{exp}_I(tJ\tilde V)^* he^{-\phi_0} =he^{-\phi_t}.$
\end{lemma}
This is essentially well-known, but as we could not find an adequate reference we include a proof here.
\begin{proof}
It is shown in \cite[Lemma 12]{do1} that it is possible to lift $V$ to a vector field $\tilde V$ on $L \to X$. Below we recall the construction of $\tilde V$ and show that one of the lifts satisfies the required properties.

By computing the curvature of both sides and using the $dd^c$ lemma, we see that for any lift there exists a smooth function $f: [0,\infty) \to \Bbb R$ such that
\begin{equation}\label{eq: generalpullback}
\textup{exp}_I(tJ\tilde V)^*he^{-\phi_0}=h e^{-\phi_t + f(t)}.
\end{equation}
We will show that for the right choice of $\tilde V$ we have $f(t) \equiv 0$. In fact, as it will be clarified below, it all depends on how we choose the Hamiltonian potential of $V$.

Suppose $v \in C^\infty(X)$ such that $i_V \omega_{\phi_0} = d v$. After perhaps adjusting $v$ by a constant, one can compute that (see \cite{ma0}, \cite[Example 4.26]{sz2}) 
$$\phi_t(x)-\phi_0(x) = 2 \int_0^t v(\textup{exp}_I(lJV)x)dl.$$
Let us fix $x_0 \in \textup{Crit}(v)$, i.e., $dv(x_0)=0$. This gives $V(x_0)=0$, hence by the above
\begin{equation}\label{eq: Kahlerlevel}
\phi_t(x_0)=\phi_0(x_0)+2tv(x_0).
\end{equation}

We now recall the main elements of \cite[Lemma 13]{b2} and its proof. For this, it will be more convenient to use the complex notation for holomorphic vector fields. To avoid confusion, recall that $V^\Bbb C = V - i JV$ and $V=\textup{Re }V:=(V^\Bbb C + \overline {V^\Bbb C})/2$.

Let $(z_1,\ldots,z_n)$ be coordinates on $X$ in a neighborhood $U$ of $x_0$. Let $s$ be a non-vanishing section of $L$ on $U$ and we introduce $e^{-\phi(z)}:=he^{-\phi_0}(s,\bar s)(z)$. Let $W^\Bbb C$ be the generator of the natural $\Bbb C^*$-action along the fibres of $L$. In local holomorphic coordinates $(z_1,\ldots,z_n,w)$ of $L \to X$ on $U$ we have $W^\Bbb C = w \frac{\del}{\del w}$ and $V^{\Bbb C}=V^j \frac{\del }{\del z_j}= -2i \phi^{j\bar k} v_{\bar k}\frac{\del }{\del z_j}$ (note the missing factor in the corresponding formula in the proof of \cite[Lemma 13]{b2}), where we have used that $V$ is Hamiltonian (in holomorphic coordinates $i V^j \phi_{j\bar k} =2 v_{\bar k}$).

If $V^\Bbb C_{hor}$ is the horizontal lift of $V^\Bbb C$ with respect to the connection of the metric $he^{-\phi_0}$ on $L$, then one can compute that $V^\Bbb C_{hor} = V^\Bbb C + \del \phi(V^\Bbb C) W^\Bbb C$. 
An elementary calculation gives that  
$$\tilde V^\Bbb C = V^\Bbb C_{hor} + 2i v W^\Bbb C=V^\Bbb C + W^\Bbb C (\del \phi(V^\Bbb C)+2iv)$$
is a holomorphic lift of $V$ to $L \to X$.
By this last formula, at the critical point $x_0$ we actually have $J \tilde V^\Bbb C(x_0)=-2v(x_0)w \frac{\del}{\del w}$. This immediately gives that the flow of $J\tilde V = \textup{Re }J\tilde V^\Bbb C$ satisfies $\textup{exp}_I(tJ\tilde V)(x_0,w )=(x_0, e^{-v(x_0)t}w)$, ultimately implying
\begin{equation}\label{eq: linebundlelevel}
\textup{exp}_I(tJ\tilde V)^*he^{-\phi_0(x_0)}(x_0)=h e^{-\phi_0(x_0) - 2t v(x_0)}(x_0).
\end{equation}
A comparison of \eqref{eq: generalpullback}, \eqref{eq: Kahlerlevel}  and \eqref{eq: linebundlelevel} gives that $f(t) \equiv 0$, finishing the proof.
\end{proof}

Lastly, we recall a result from \cite{b} that was the last important element in the proof of Theorem \ref{thm: Kpolystable}:

\begin{proposition}\label{prop: Berman_prop} Let $X$ be a K\"ahler manifold with positive line bundle $L \to X$ and a normal test configuration $(\mathcal X,\mathcal L,\pi,\rho)$ with $S^1$-invariant smooth background metric $h$, and reduced central fiber $X_0$. Given $\phi_0 \in \mathcal H_{\Theta(h_1)}$, suppose that the associated geodesic ray $t \to \phi_t$ is induced by a vector field $\tilde V$ of $L \to X$, i.e., $\textup{exp}_I(tJ\tilde V)^*h_1e^{-\phi_0}=h_1e^{-\phi_t}$. Then  $\tilde V$ is the generator of a  $\Bbb C^*$-action, and
 $\mathcal X$ is isomorphic to $X \times \Bbb C$.
\end{proposition}

This proposition is contained in \cite[Lemma 3.4]{b}. Strictly speaking, the  statement above (that $\mathcal X$ is isomorphic to $X \times \Bbb C$) does not appear explicitly in the statement of \cite[Lemma 3.4]{b}, but the proof of  \cite[Lemma 3.4]{b} does establish the isomorphism in question (as pointed out after formula (3.16) in \cite{b}).   

\let\OLDthebibliography\thebibliography 
\renewcommand\thebibliography[1]{
  \OLDthebibliography{#1}
  \setlength{\parskip}{1pt}
  \setlength{\itemsep}{1pt plus 0.3ex}
}

\noindent{\sc Chalmers University of Technology}

\noindent{\tt robertb@chalmers.se} \vspace{0.1in}

\noindent{\sc University of Maryland}

\noindent{\tt tdarvas@math.umd.edu}\vspace{0.1in}

\noindent{\sc Laboratoire de Math\'ematiques d'Orsay,
 Univ. Paris-Sud,
 CNRS, Universit\'e Paris-Saclay,
  91405 Orsay, France}
  
\noindent{\tt hoang-chinh.lu@u-psud.fr}
\end{document}